\documentclass[11pt,reqno]{amsart}

\usepackage{amsmath,amssymb}
\usepackage{graphicx}
\usepackage{algorithm}
\usepackage{enumerate}
\usepackage{tikz}
\usetikzlibrary{decorations.pathmorphing}
\usetikzlibrary{decorations.markings}
\usetikzlibrary{positioning}
\usetikzlibrary{shapes}

\usepackage{tabu}
\usepackage{url}

\usepackage{mathrsfs}
\usepackage{caption} 
\usepackage{subcaption}

\usetikzlibrary{matrix, calc, arrows}
\usetikzlibrary{graphs}
\usetikzlibrary{graphs.standard}
\usepackage{tkz-berge}

\usepackage{todonotes}

\usepackage{ifthen}

\newtheorem{theorem}{Theorem}[section]
\newtheorem{corollary}[theorem]{Corollary}
\newtheorem{proposition}[theorem]{Proposition}

\newtheorem{question}{Question}

\theoremstyle{definition}
\newtheorem{example}[theorem]{Example}
\newtheorem{definition}[theorem]{Definition}

\newcommand{\rmv}[1]{}

\definecolor{CUorange}{RGB}{246,103,51}
\definecolor{CUpurple}{RGB}{82,45,128}
\definecolor{CUred}{RGB}{162,80,22}
\definecolor{CUgreen}{RGB}{86,97,39}
\definecolor{CUblue}{RGB}{58,73,88}

\title{$\beta$-Packing Sets in Graphs}
\author[Case, Haithcock, Laskar]{Benjamin M. Case \and Evan M. Haithcock \and Renu C. Laskar}
\date{May 31, 2019}
\address{School of Mathematical and Statistical Sciences, Clemson University, SC, USA}
\thanks{Benjamin M. Case was partially supported by the  National Science Foundation  under grants DMS-1403062 and DMS-1547399.}

\begin{document}

	\begin{abstract}
	A set $S\subseteq V$ is $\alpha$\textit{-dominating} if for all $v\in V-S$, $|N(v) \cap S | \geq \alpha |N(v)|.$ The $\alpha$\textit{-domination number} of $G$ equals the minimum cardinality of an $\alpha$-dominating set $S$ in $G$.  Since being introduced by Dunbar, et al. in 2000, $\alpha$-domination has been studied for various graphs and a variety of bounds have been developed. In this paper, we propose a new parameter derived by flipping the inequality in the definition of $\alpha$-domination. We say a set $S \subset V$ is a \textit{$\beta$-packing set} of a graph $G$ if $S$ is a proper, maximal set having the property that for all vertices $v \in V-S$, $|N(v) \cap S| \leq \beta |N(v)|$
     for some $0 < \beta \leq 1.$ The \textit{$\beta$-packing number} of $G$ ($\beta$-pack($G$)) equals the maximum cardinality of a $\beta$-packing set in $G$.   In this research, we determine $\beta$-pack($G$) for several classes of graphs, and we explore some properties of $\beta$-packing sets.

  \medskip

  \noindent Keywords: $\beta$-packing, $\alpha$-domination,  graph theory, graph parameters
\end{abstract}
\maketitle
\section{Introduction}
	Let $G=(V,E)$ be a graph with vertex set $V=\{v_1,v_2,...,v_n\}$ and \emph{order} $n = |V|$. The {\em open neighborhood} of a vertex $v$ is the set $N(v) := \{u\: |\: uv \in E\} $ of vertices $u$ that are adjacent to $v$; the \emph{closed neighborhood} of $v$, $N[v]:=N(v)\cup \{v\}.$ 
	
	
	A set $S\subseteq V$ is $\alpha$\textit{-dominating} if for all $v\in V-S$, $|N(v) \cap S | \geq \alpha |N(v)|.$ The $\alpha$\textit{-domination number} of $G$ equals the minimum cardinality of an $\alpha$-dominating set $S$ in $G$.  Since being introduced by Dunbar, Hoffman, Laskar, and Markus \cite{Dunbar} in 2000, $\alpha$-domination has been studied for various graphs and a variety of bounds have been developed, see \cite{dahme2004some,Rubalcaba,gagarin2009upper,rad15,Das2018}. In this paper, we present a new parameter that is motivated by flipping the inequality in $\alpha$-domination, known as the $\beta$-packing set. 
	
	\begin{definition}
	For a graph $G = (V, E)$, a set $S \subset V$ is a \textit{$\beta$-packing set} of a graph $G$ if $S$ is a proper, maximal set having the property (which we call the \emph{$\beta$-packing property}) that for all vertices $v \in V-S$, \[|N(v) \cap S|\leq \beta |N(v)| \]
     for some $0 < \beta \leq 1.$ The \textit{$\beta$-packing number} of $G$, $\beta$-pack($G$), equals the maximum cardinality of a $\beta$-packing set in $G$.
	\end{definition}
	
	For example, we say that a set $S \subset V$ is a 1/2-beta packing set if $v \in V -S$, $\frac{|N(v) \cap S|}{|N(v)|} \leq$ 1/2 and is maximal. The 1/2-beta packing number equals the maximum cardinality of a 1/2-beta packing set in $G$.\medskip
	
	\begin{example}
	In Figure \ref{fig:clemsongraph} we show all of the 1/2-beta packing sets of the shown graph (up to symmetry). The $\beta$-packings sets are shown as the black filled vertices. Note that in each graph, no subset of $V-S$ can be added to $S$ while preserving both the $\beta$-packing property and keeping the $\beta$-packing set a proper subset. The largest cardinality of these sets is 2, so $\frac{1}{2}\beta$-pack(G) = 2.  
	\end{example}
	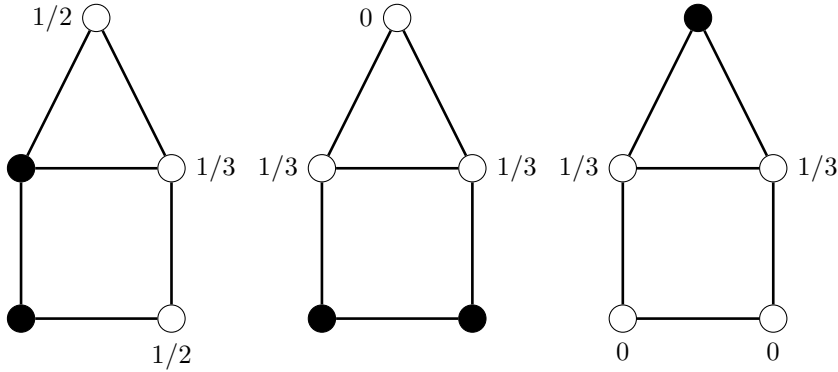
\begin{figure}[!tbh]
	    \centering
	    \begin{tikzpicture}
			\node [draw,circle, label=below:{}, fill] (A0) at (0,0) {};
			\node [draw,circle, label=below:{}, fill] (A1) at (2,0) {};
			\node [draw,circle, label=left:{\small1/3}] (A2) at (0,2) {};
			\node [draw,circle, label=right:{\small1/3}] (A3) at (2,2) {};
			\node [draw,circle, label=left:{\small0}] (A4) at (1,4) {};
			
			\draw [line width=1pt] (A0) edge (A2);
			\draw [line width=1pt] (A0) edge (A1);
			\draw [line width=1pt] (A1) edge (A3);
			\draw [line width=1pt] (A2) edge (A3);
			\draw [line width=1pt] (A4) edge (A2);
			\draw [line width=1pt] (A4) edge (A3);
			
			\node [draw,circle, label=below:{},fill ] (B0) at (-4,0) {};
			\node [draw,circle, label=below:{\small1/2}, ] (B1) at (-2,0) {};
			\node [draw,circle, label=left:{}, fill] (B2) at (-4,2) {};
			\node [draw,circle, label=right:{\small1/3}] (B3) at (-2,2) {};
			\node [draw,circle, label=left:{\small1/2}] (B4) at (-3,4) {};
			\draw [line width=1pt] (B0) edge (B2);
			\draw [line width=1pt] (B0) edge (B1);
			\draw [line width=1pt] (B1) edge (B3);
			\draw [line width=1pt] (B2) edge (B3);
			\draw [line width=1pt] (B4) edge (B2);
			\draw [line width=1pt] (B4) edge (B3);
			
			\node [draw,circle, label=below:{\small0}, ] (C0) at (4,0) {};
			\node [draw,circle, label=below:{\small0}, ] (C1) at (6,0) {};
			\node [draw,circle, label=left:{\small 1/3}, ] (C2) at (4,2) {};
			\node [draw,circle, label=right:{\small1/3}] (C3) at (6,2) {};
			\node [draw,circle, label=left:{},fill] (C4) at (5,4) {};
			\draw [line width=1pt] (C0) edge (C2);
			\draw [line width=1pt] (C0) edge (C1);
			\draw [line width=1pt] (C1) edge (C3);
			\draw [line width=1pt] (C2) edge (C3);
			\draw [line width=1pt] (C4) edge (C2);
			\draw [line width=1pt] (C4) edge (C3);
			\end{tikzpicture}
	    \caption{The 1/2-beta packing sets (up to symmetry), shown in black.  $\frac{1}{2}\beta$-pack(G) = 2.}
	    \label{fig:clemsongraph}
	\end{figure}

 
\section{Examples and $\beta$-Packing Sets for Classes of Graphs}
To begin we will consider some examples of different classes of graphs and try to determine some patterns about the $\beta$-packing number.  We start by looking at the 1/2-beta packing sets for paths and then generalize these results to all paths and cycles. A $\frac{1}{2}\beta$-packing for $P_6$ is show in Figure \ref{path}. 

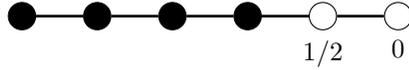
\begin{figure}[!tbh]
    \begin{center}
		\begin{tikzpicture}
				\node [draw,circle, label=below:{}, fill] (P1) at (0,0) {};
				\node [draw,circle, label=below:{},fill] (P2) at (1,0) {};
				\node [draw,circle, label=below:{}, fill] (P3) at (2,0) {};
				\node [draw,circle, label=below:{}, fill] (P4) at (3,0) {};
				\node [draw,circle, label=below:{\small1/2}] (P5) at (4,0) {};
				\node [draw,circle, label=below:{\small0}] (P6) at (5,0) {};
				\draw [line width=1pt] (P1) edge (P2);
				\draw [line width=1pt] (P2) edge (P3);
				\draw [line width=1pt] (P3) edge (P4);
				\draw [line width=1pt] (P4) edge (P5);
				\draw [line width=1pt] (P5) edge (P6);
			
		\end{tikzpicture}
    \end{center}
    \caption{The $\frac{1}{2}\beta$-packing set of a path, $P_6$.}
    \label{path}
\end{figure}
\begin{proposition} \label{paths-1/2}
	Given a path $P_n$ of length n $\geq$ 2,
	$\frac{1}{2}\beta\text{-pack}(P_n) = n - 2$ and $V-S$ is connected.
\end{proposition}
\begin{proof}
Consider a path of length $n$, $P_n = (V, E).$ If $V-S$ is not connected, $S$ is not maximal, see Proposition \ref{Prop:connected} where we show this in general. Suppose $S \subset V$ and $\{v_i, v_{i+1}\} = V - S$ for some $i\in[1, n-1]$. As $S$ is proper, it suffices to show that the $\beta$-packing property is fulfilled and that S is maximal. To show the former, consider the following cases:
\begin{itemize}
    \item If $\deg(v_i) = 1$, then $N(v_i) \cap S = \emptyset$ and \[ \frac{|N(v_i) \cap S|}{\deg(v_i)} = 0 \leq \frac{1}{2} .\]
    \item If $\deg(v_i) = 2$, then $N(v_i) \cap S = \{v_{i-1}\}$ and\[ \frac{|N(v_i) \cap S|}{\deg(v_i)} = \frac{1}{2} \leq \frac{1}{2}.\]
    \item If $\deg(v_{i+1}) = 1$ then $N(v_{i+1}) \cap S = \emptyset$ and \[ \frac{|N(v_i) \cap S|}{\deg(v_i)} = 0 \leq \frac{1}{2}.\] 
    \item If $\deg(v_{i+1}) = 2$ then $N(v_{i+1}) \cap S = \{v_{i+2}\}$ and \[ \frac{|N(v_i) \cap S|}{\deg(v_i)} = \frac{1}{2} \leq \frac{1}{2}.\]
\end{itemize}
		Thus the $\beta$-packing property holds in all cases.
		Now, we need to show that $S$ is maximal. WLOG, suppose $V-S = \{v_i\}$. We will again consider cases:
		\begin{itemize}
		    \item If $\deg(v_i) = 1$ then $ N(v_i) \cap S = \{v_{i+1}\}$ and \[ \frac{|N(v_i) \cap S|}{\deg(v_i)} = 1 > \frac{1}{2}.\]
		    \item If $\deg(v_i) = 2$ then $ N(v_i) \cap S = \{v_{i-1}, v_{i+1}\}$ \[ \frac{|N(v_i) \cap S|}{\deg(v_i)} = 1 > \frac{1}{2}. \]
		\end{itemize}
\end{proof}
The following three results cover all the possible values of $\beta$ and show what the corresponding value of $\beta$-pack$(P_n)$ is. 
\begin{proposition}
        For $\frac{1}{2} \leq \beta < 1$ and $n\geq 2$, $\beta \text{-pack}(P_n) = n-2$ and $V-S$ is connected.
\end{proposition}

\begin{proof}
This follows the same proof as the $\frac{1}{2}\beta$-packing set.
\end{proof}

\begin{proposition} \label{Pnbetasmall}
For $0 < \beta < \frac{1}{2}$, $\beta \text{-pack}(P_n) = 0$.
\end{proposition}

\begin{proof}
For any $v_i \in V$, $\deg(v_i) =$ 1 or 2. This implies
$ \frac{|N(v_i) \cap S|}{\deg(v_i)}$ is either  0, $\frac{1}{2}$ or 1. But $\frac{|N(v_i) \cap S|}{\deg(v_i)} \leq \beta < \frac{1}{2}$, which implies
$S = \emptyset$. So, $\beta \text{-pack}(P_n) = 0$.
\end{proof}

\begin{proposition}
For $\beta = 1$, $\beta \text{-pack}(P_n) = n-1$.
\end{proposition}

\begin{proof}
Letting $\beta = 1$ means that for any $v_i$, $\frac{|N(v_i) \cap S|}{\deg(v_i)} \leq \beta$. As $S$ must be a proper subset, we have to leave one node out of $S$. Thus,\\
$\beta \text{-pack}(P_n) = n - 1$.
\end{proof}

\begin{corollary}
Given a cycle $C_n$ of size $n \geq 3$, \[\beta\text{-pack}(P_n) = 
    \begin{cases}
        0 & 0 < \beta < \frac{1}{2}\\
        n - 2 & \frac{1}{2} \leq \beta < 1\\
        n - 1 & \beta = 1\\
    \end{cases}
    \] and $V-S$ is connected.
\end{corollary}

\begin{proof}
Note that any path can be made into a cycle by adding an edge. Thus, the proof for a cycle is identical to that of a path except that we need only to consider the cases of degree 2.
\end{proof}

Next we will consider complete bipartite graphs and determine their $\beta$-packing numbers. An example of a $\beta$-packing set is shown in Figure \ref{bipartite} for $K_{4,5}$. 

\begin{proposition}\label{Kmn}
		Let $K_{m,n} = (V_m, V_n, E)$ be a complete bipartite graph. Then for $\beta < 1$, all $\beta$-packing sets $S\cup S'$ have the same size, where $S \subset V_m \subset V,$ $S' \subset V_n\subset V$,  with $|S| = \lfloor \beta \cdot m \rfloor$ and $|S'| = \lfloor \beta \cdot n \rfloor $.  Thus, $\beta\text{-}pack(K_{m,n}) = \lfloor \beta \cdot m \rfloor + \lfloor \beta \cdot n \rfloor.$
\end{proposition}

\begin{figure} 
    \centering
    \begin{tikzpicture}
    \node [draw,circle, label=below:{\small1/2}] (M0) at (0,0) {};
    \node [draw,circle,fill] (M1) at (1,0) {};
    \node [draw,circle,label=below:{\small1/2}] (M2) at (2,0) {};
    \node [draw,circle,label=below:{\small1/2}] (M3) at (3,0) {};
    \node [draw,circle,fill] (M4) at (4,0) {};
    \node [draw,circle,fill] (N0) at (0.5,2) {};
    \node [draw,circle,fill] (N1) at (1.5,2) {};
    \node [draw,circle,label={\small 2/5}] (N2) at (2.5,2) {};
    \node [draw,circle,label={\small 2/5}] (N3) at (3.5,2) {};
    
    \foreach \x in {0,1,2,3,4}{
    \foreach \y in {0,1,2,3}{
    \draw [line width=1pt] (M\x) edge (N\y);}}
    \end{tikzpicture}
    \caption{A possible $\frac{1}{2}\beta$-packing set of the complete bipartite graph, $K_{4,5}$.}
    \label{bipartite}
\end{figure}
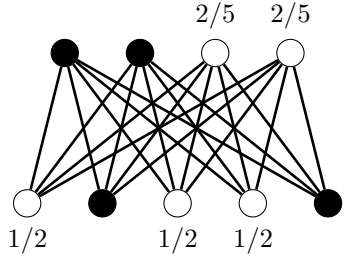

\begin{proof}
Let $S \subset V_m$ be any subset of size $\lfloor \beta \cdot m \rfloor$ and $S'\subset V_n$ be any subset of size $\lfloor \beta \cdot n \rfloor$. 
Since $\beta < 1$, $\lfloor \beta \cdot m \rfloor < m$ and $\lfloor \beta \cdot n \rfloor < n$. Thus, $S \cup S'$ is proper. It suffices to show that the $\beta$-packing property is fulfilled and that $S \cup S'$ is maximal. Let $v \in V_m - S$. Then, $\deg(v) = n$ implies \[|N(v) \cap S'| = |S'| = \lfloor \beta \cdot n \rfloor. \]
Thus,
\[  \frac{|N(v) \cap (S \cup S')|}{|N(v)|} = \frac{|N(v) \cap (S')|}{|N(v)|} = \frac{\lfloor \beta \cdot n \rfloor}{n} \leq \frac{\beta \cdot n}{n} = \beta. \]
Now let $v' \in V_n - S'$. Then, following the same process, we see that
$$\frac{|N(v') \cap (S \cup S')|}{|N(v')|} \leq \beta.$$
Finally, we must show that $S \cup S'$ is maximal. Suppose for contradiction there was a proper subset $S\cup S' \cup U \subset V$ for which the $\beta$-packing property held with $\emptyset \neq U \subset V -(S\cup S')$.  $U$ must contain at least one vertex $u$. WLOG, let $u$ be in the side $u \in V_m \cap U.$ Then for $v \in V_n - (S'\cup U)$, 
\[ \frac{|N(v) \cap (S \cup S'\cup U)|}{|N(v)|} \geq \frac{|N(v) \cap (S \cup \{ u \} )|}{|N(v)|} = \frac{\lfloor \beta \cdot m \rfloor +1 }{m}. \]
Note that $\beta \cdot m < \lfloor \beta \cdot m \rfloor +1$ which implies $\beta < \frac{\lfloor \beta \cdot m \rfloor +1}{m}$. Thus 
\[\frac{|N(v) \cap (S \cup S'\cup U)|}{|N(v)|} > \beta, \] 
so $S\cup S'$ is maximal and $\beta \text{-pack}(K_{m,n}) = \lfloor \beta \cdot m \rfloor + \lfloor \beta \cdot n \rfloor$.

\end{proof}

\begin{proposition} \label{Kmnbeta1}
        For $\beta = 1$, $\beta \text{-pack}(K_{m,n}) = m+n-1$.
\end{proposition}
\begin{proof}
As the $\beta$-packing set must be proper, we let all the nodes be in the $\beta$-packing set and then remove one.
As $|K_{m,n}| = m+n$, $\beta \text{-pack}(K_{m,n}) < |K_{m,n}|$. Adding another node to this set would be all of $K_{m,n}$, so $S \cup S'$ is both proper and maximal.\\
Let $v \notin S \cup S'$. Then, $\frac{|N(v) \cap (S \cup S')|}{|N(v)|} = 1$, since every other node is in $S \cup S'$. Thus, $\beta \text{-pack} (K_{m,n}) = m+n-1$.
\end{proof}

If we try to generalize these results to complete multipartite graphs, Proposition \ref{Kmn} does not generalize in the natural way, but Proposition \ref{Kmnbeta1} does.

\begin{example}
Consider the complete multipartite graph $K_{3,3,3,3}$ and let $\beta = 1/2$. A $\beta$-packing set is given by taking 1 vertex in each of three partitions and 2 vertices out of the forth partition, for a total of 5 vertices in $S$.  One can check this gives 
\[ \frac{1}{2}\beta\text{-pack}(K_{3,3,3,3})=5 > \lfloor \beta \cdot 3 \rfloor + \lfloor \beta \cdot 3 \rfloor + \lfloor \beta \cdot 3 \rfloor +\lfloor \beta \cdot 3 \rfloor =4. \]
\end{example}

\begin{corollary} 
        For $\beta = 1$, $\beta \text{-pack}(K_{n_1,n_2,...,n_m}) = n_1+\cdots + n_m-1$.
\end{corollary}
\begin{proof}
The proof is similar to the bipartite case. 
\end{proof}

\section{ General Properties of $\beta$-packing sets}
In this section we present several general properties about $\beta$-packing sets and the $\beta$-packing number.   Our first property shows how the $\beta$-packing numbers corresponding to different $\beta$'s are related.  
\begin{proposition}
		Let $0 < \beta_1 \leq \beta_2 \leq 1$. Then $\beta_1$-pack($G$) $\leq \beta_2$-pack($G$). 
	\end{proposition}
	\begin{proof}
		Consider $\beta_1$-pack($G$), for any $\beta_1$-packing set $S$, $\forall v \in V-S$, 
		\[  \frac{|N(v) \cap S |}{|N(v)| } \leq \beta_1 \leq \beta_2. \]
		So any such $S$ is contained in a $\beta_2$-packing set and one could add vertices until $S$ becomes maximal w.r.t $\beta_2$. 
	\end{proof}

    It was already seen in Proposition \ref{paths-1/2} for paths that the complement a $\beta$-packing set is connected. This is in fact a general property that holds for all graphs. 
	
	\begin{proposition}\label{Prop:connected}
		For any $\beta$-packing set $S$, $V-S$ is connected. 
	\end{proposition}
	\begin{proof}
		If $V-S$ is not connected, then $S$ is not maximal since one of the components of $V-S$ could be added to $S$ to form $S'$ and for all other $v\in V-S'$ we still have the property
		\[ \frac{|N(v) \cap S'|}{|N(v)| } \leq \beta \]
		and $S'$ would still be proper. 
	\end{proof}

			\begin{proposition}
		Let $\Delta (G)$ be the max degree of a vertex of a connected graph. If $\beta < \frac{1}{\Delta(G)}$, then $\beta$-pack$(G)=0$. 
	\end{proposition}
	\begin{proof}
		Suppose $S$ is a nonempty $\beta$-packing set. For any vertex $v \in V-S$, if a neighbor is in a $\beta$-packing set $S$, then 
		\[ \beta < \frac{1}{\Delta(G)} \leq \frac{1}{\deg(v)} \leq \frac{|N(v) \cap S |}{\deg(v) }, \]
		a contradiction.  Thus no vertex has a neighbor in $S$. Therefore $S = \emptyset$. 
	\end{proof}
	
	The next three properties investigate the question of which values for $\beta$ in the interval $0< \beta \leq 1$ are interesting to consider. 
	\begin{proposition}
		If $\beta = 1$, then $\beta$-pack$(G)=n-1$. 
	\end{proposition}
	\begin{proof}
		A $\beta$-packing set must be proper, but we can just leave out any one vertex.  
	\end{proof}

	\begin{proposition}
		Let $G$ be connected. If $\beta < 1$, then $\beta$-pack$(G) < n-1$. 
	\end{proposition}
	\begin{proof}
		 Suppose $\{v\} = V- S$. Then 
		\[ \frac{|N(v)\cap S|}{|N(v)|} = \frac{|N(v)|}{|N(v)|} =1 > \beta. \]
	\end{proof}
	    Let us consider the following question a bit more. 
	    \begin{question}
	    Given a graph $G$ how many "interesting" $\beta$'s are there to consider? 	 By interesting we mean that as $\beta$ increases from 0 to 1 it is only at these values where the value of $\beta$-pack$(G)$ could change. 
	    \end{question}
	
		Let $\delta(G) = d_1, ..., d_t = \Delta(G)$ be the distinct degrees of vertices in the graph. Then we claim the possible interesting $\beta$'s are
a subset of the following ratios: 		\[0,\frac{1}{d_1}, \frac{2}{d_1}, ..., \frac{d_1-1}{d_1},1\]
		\[\frac{1}{d_2}, \frac{2}{d_2}, ..., \frac{d_2-1}{d_2}\]
		\[ \vdots \hspace{2.1cm} \vdots \]
		\[\frac{1}{d_t}, \frac{2}{d_t}, ..., \frac{d_t-1}{d_t}.\]

\section{Related Parameters}
    The initial motivation for defining $\beta$-packing sets was from $\alpha$-domination, and it is natural to ask what relationships the two parameters may have with each other. One might ask if $\beta = \alpha$ weather 
    \[\gamma_\alpha(G) \leq \beta\text{-pack}(G)? \quad \text{ or }  \quad  \gamma_\alpha(G) \geq \beta\text{-pack}(G)? \]   The answer is neither one in general. We have by Proposition \ref{Pnbetasmall} that $\beta = \frac{1}{3}$-pack$(P_n) = 0$. But from \cite[Prop. 1]{Dunbar} that $\gamma_{\alpha = \frac{1}{3} }(P_n) = \lceil \frac{n}{3} \rceil$.  So this is an example of $\frac{1}{3}\text{-pack}(G) < \gamma_{\frac{1}{3}}(G)$. 
    
    On the other hand, we have that by Proposition \ref{Kmn} that if when $\beta <1$ $\beta\text{-pack}(K_{m,n}) = \lfloor \beta \cdot m \rfloor + \lfloor \beta \cdot n \rfloor$. In \cite[Prop. 4]{Dunbar} they have the result that for $1\leq m \leq n$
    \[   \gamma_\alpha(K_{m,n} ) = \min \{ \lceil \alpha m \rceil + \lceil \alpha n \rceil, m  \}.\]
	Thus if we let for example $m =1$, $n =10$, $\beta = \alpha = 1/2$ we get than
	\[ \gamma_{\frac{1}{2}}(K_{1,10}) = 1 < \frac{1}{2}\text{-pack}(K_{1,10}) =5. \]

	We think it is an interesting open direction of study to consider if there are different relationship between $\alpha$-domination and $\beta$-packing would be interesting to consider. 
    
    \section{Conclusion}
    In conclusion, we have introduced the new graph parameter, the $\beta$-packing number, and studied some of its properties and given formulas for it for certain classes of graphs. Our motivation for defining $\beta$-packing sets comes for $\alpha$-domination, but we leave it as an open direction to investigate what relationships these two parameters have with each other.  Other interesting open directions would include determining the value of the $\beta$-packing number for other classes of graphs and determining the computational complexity of finding $\beta$-packing sets or the $\beta$-packing number. We hope that this introductory paper and promising future directions will promote further interest in considering $\beta$-packing.

	\nocite{Domke,gera2016graph}
	\bibliographystyle{plain}
	\bibliography{references}
\end{document}